\documentclass{article}
\usepackage[latin1]{inputenc}
\usepackage[english]{babel} 
\usepackage{amssymb,amsmath,amsthm,enumerate}
\usepackage{vmargin}
\usepackage{dcolumn}
\usepackage{color}
\usepackage{bm}
\usepackage{enumitem}
\usepackage{amssymb}
\usepackage{dirtytalk}
\usepackage{lipsum}
\usepackage{amsthm}
\usepackage{fancyhdr}
\usepackage{makecell}
\usepackage{array}
\usepackage[usenames,dvipsnames]{xcolor}
\usepackage{graphicx}
\usepackage{tcolorbox}
\usepackage{amsfonts}
\usepackage{commath}
\usepackage{mathtools}
\usepackage{amssymb}
\usepackage{breqn}
\usepackage{multirow}
\usepackage{color}
\setmargrb{2.2 cm}{1.7 cm}{2.23 cm}{3.1 cm}
\usepackage{amsmath}
\usepackage[normalem]{ulem}

\newtheorem{theorem}{Theorem}

\newtheorem{proposition}[theorem]{Proposition}
\newtheorem{lemma}[theorem]{Lemma}

\usepackage[useregional]{datetime2}

\newcommand{\preal}{\real_+}
\newcommand{\bb}{\mathcal{B}}
\newcommand{\kay}{\mathcal{K}}
\newcommand{\half}{\frac{1}{2}}

\usepackage{titlesec}
\titleformat{\subsection}[runin]
  {\normalfont\large\bfseries}{\thesubsection}{1em}{}

   \allowdisplaybreaks
   
\begin{document}
 \title{Fredholm Property of the Linearized Boltzmann Operator for a Polyatomic Single Gas Model}


          \author {St\'ephane Brull \thanks {Univ. Bordeaux, IMB, UMR 5251, F- 33400 TALENCE, FRANCE, (Stephane.Brull@math.u-bordeaux1.fr)}  , \and  \hspace{-0.75cm} Marwa Shahine \thanks {Univ. Bordeaux, IMB, UMR 5251, F- 33400 TALENCE, FRANCE, (Marwa.Shahine@math.u-bordeaux.fr)}  , \and \hspace{-0.75cm} and  Philippe Thieullen \thanks{Univ. Bordeaux, IMB, UMR 5251, F- 33400 TALENCE, FRANCE, (Philippe.Thieullen@math.u-bordeaux.fr).}
         }

         \pagestyle{myheadings} 

\date{}         
\maketitle

 \begin{abstract}
 In the following work, we consider the Boltzmann equation that models a polyatomic gas by representing the microscopic internal energy by a continuous variable I. Under some convenient assumptions on the collision cross-section $\mathcal{B}$, we prove that the linearized Boltzmann operator $\mathcal{L}$ of this model is a Fredholm operator. For this, we write   $\mathcal{L}$ as a  perturbation of the collision frequency multiplication operator, and we prove that the perturbation operator $\kay$ is compact. The result is established after inspecting the kernel form of $\kay$ and proving it to be $L^2$ integrable over its domain using elementary arguments.
          \end{abstract}

 \section{Introduction}\label{intro}

 \indent 
 This work is devoted to the study of the Fredholm property of the linearized Boltzmann operator $\mathcal{L}$ for polyatomic gases. In fact, the Fredholm property  of $\mathcal{L}$ is essential in the Chapman-Enskog asymptotics, and has been always assumed for polyatomic gases in the literature, \cite{brull,11} for instance. In contrast to monoatomic molecules, the energy in a polyatomic gas is not totally stored in the kinetic energy of its molecules but also in their rotational and vibrational modes. Within the kinetic theory, the microscopic internal energy was integrated in the models by several ways. For example, the internal energy is assumed to take discrete values in \cite{ bisic, 11,cc, 10}. On the other hand, some kinetic models take another path by representing the microscopic internal energy by a continuous parameter, and modeling the Boltzmann equation \cite{9} using the Borgnakke-Larsen procedure \cite{8}. 
 Many formal results were achieved for  the model in \cite{9}. Using \cite{28}, the Chapman-Enskog method was recently developed  in \cite{brull}, and many macroscopic models of extended thermodynamics were derived \cite{31}. 
 \vspace{1mm}

 Due to the complexity of the Boltzmann collision integral in monoatomic gases, simple models for the collision term were developed. One of the remarkable simplified models is the BGK model introduced in \cite{bgk}. Yet, the construct of this model gives a Prandtl
number equal to 1 and cannot predict the correct transport coefficients. This inconvenience was solved by the ES-BGK model \cite{holway} that includes another parameter to uncouple the thermal relaxation from the viscosity
relaxation. Besides, the kinetic theory of polyatomic gases is even more complicated. 
 In this context, the simplified ES-BGK  model has been also developed \cite{30,22}, where the return to equilibrium of the solutions in the homogeneous case has been studied as well in \cite{22}. In \cite{yun}, the existence result of the ES-BGK model was achieved in the case where the solution lies close to equilibrium.
 
  
\vspace{1mm}

 In order to construct a quantitative theory and to obtain explicit convergence rates to the equilibrium, explicit estimates for the spectral gap of the linearized Boltzmann operator
should be obtained. 
In the monoatomic single gas setting, Grad \cite{grad} proved that the linearized Boltzmann operator is a Fredholm operator, by writing it as compact perturbation of a coercive multiplication operator. Bobylev was
the first to find explicit estimates for the spectral gap for Maxwell molecules by implementing Fourier methods \cite{b1,b2}. Using the spectral gap of Maxwell molecules, coercivity estimates for hard sphere gases were recently established in \cite{bm}, and for monoatomic gases without angular cut-off in \cite{mouhotstrain}. For monoatomic mixtures, $\mathcal{L}$ was proved to be Fredholm in \cite{milana2}. Right after, coercivity estimates on the spectral gap of the linearized Boltzmann operator were obtained \cite{daus}. For a single diatomic gas, the operator $\kay$ was proved to be a Hilbert-Schmidt operator in \cite{us} under some assumption on the collision cross-section. In \cite{Niclas}, $\mathcal{L}$ was proved to be Fredholm for a single polyatomic gas using a different approach, and in \cite{borsonii}, the Fredholm property of $
\mathcal{L}$ was proved for polyatomic gases undergoing resonance. In this work, we aim to generalize the work \cite{us}  for a single polyatomic gas using elementary arguments. For this, we write $\mathcal{L}$ as a compact perturbation of the collision frequency multiplication operator, and we prove as well that the collision frequency $\nu$ is coercive. This implies that $\mathcal{L}$ is a Fredholm operator.\vspace{1mm}

\indent The plan of the paper is as follows: In section \ref{Sec.2}, we give a brief recall on the collision model \cite{9}, which describes the microscopic state of polyatomic gases, and give an equivalent formulation of the collision operator. In section \ref{Sec.3}, we define the linearized operator $\mathcal{L}$, which is obtained by approximating the distribution function $f$ around the Maxwellian $M$. The main aim of this paper is to prove that the linearized Boltzmann operator is a Fredholm operator, which is achieved in section \ref{Sec.4}. In particular, we write $\mathcal{L}$ as $\mathcal{L}=\kay-\nu$ Id and we prove that $\kay$ is compact, and $\nu$ is coercive. As a result, $\mathcal{L}$ is viewed as a compact perturbation of the multiplication operator $\nu$ Id. To prove $\mathcal{K}$ is compact, we write $\mathcal{K}$ as $\mathcal{K}_3+\mathcal{K}_2-\mathcal{K}_1$, and we prove each $\mathcal{K}_{i}$, with $i=1,\cdots,3$, to be a Hilbert-Schmidt operator. In section 5, we prove the coercivity and we give the monotony property of the collision frequency, which helps to locate the essential spectrum of $\mathcal{L}$. 
\vskip2mm

 \section{The Classical model}\label{Sec.2}

We present first the model in \cite{9} on which our work is mainly based. We start with physical conservation equations and proceed by parameterizing them.

 \subsection{Boltzmann Equation}\label{Subsec.1 }
   Without loss of generality, we first assume that the molecule mass equals unity, and we denote as usual by $(v,v_*)$, $(I,I_*)$ and $(v',v'_*)$, $(I',I'_*)$ the pre-collisional and post-collisional velocity and energy pairs respectively. In this model, the internal energies are assumed to be continuous.  The following conservation of momentum and total energy equations hold:
 
  \begin{align}
  v+v_*&=v^{\prime}+v^{\prime}_* \\
 \frac{1}{2}v^2+\frac{1}{2}v_*^2+I+I_*&=\frac{1}{2}v^{\prime2}+\frac{1}{2}v^{\prime2}_*+I^{\prime}+I^{\prime}_*.\label{conservation}
  \end{align}
 From the above equations, we can deduce the following equation representing the conservation of total energy in the center of mass reference frame:
\begin{equation*}
    \frac{1}{4}(v-v_*)^2+I+I_*=\frac{1}{4}(v^{\prime}-v^{\prime}_*)^{2}+I'+I'_*=E,
\end{equation*}
\noindent with $E$ denoting the total energy of a particle. We introduce in addition the parameter  $R\in[0,1]$ which represents the portion allocated to the kinetic energy after collision out of the total energy, and the parameter $r\in[0,1]$ which represents the distribution of the post internal energy among the two interacting molecules. Namely,
\begin{equation*}
\begin{aligned}
   \frac{1}{4}(v^{\prime}-v^{\prime}_*)^{2}&=RE\\
    I^{\prime}+I^{\prime}_*&=(1-R)E,
\end{aligned}
\end{equation*}
and 
\begin{equation*}
\begin{aligned}
    I'&=r(1-R)E\\
    I'_*&=(1-r)(1-R)E.
\end{aligned}
\end{equation*}
Using the above equations, we can express the post-collisional velocities in terms of the other quantities by the following 
\begin{equation*}
\begin{aligned}
    v'\equiv v'(v,v_*,I,I_*,\omega,R)&=\frac{v+v_*}{2}+\sqrt{{RE}}\,\,T_\omega\bigg[\frac{v-v_*}{|v-v_*|}\bigg]\\
v'_*\equiv v'_*(v,v_*,I,I_*,\omega,R)&=\frac{v+v_*}{2}-\sqrt{RE}\,\,T_\omega\bigg[\frac{v-v_*}{|v-v_*|}\bigg],
\end{aligned}
\end{equation*}
where $\omega\in S^2$, and  $T_\omega(z)=z-2(z.\omega)\omega$. In addition, we define the parameters $r'$ and $R'\in[0,1]$ for the pre-collisional terms in the same manner as $r$ and $R$. In particular

 \begin{equation*}
\begin{aligned}
    \frac{1}{4}(v-v_*)^2&=R'E\\
    I+I_*&=(1-R')E,
\end{aligned}
\end{equation*}
and 
\begin{equation*}
\begin{aligned}
    I&=r'(1-R')E\\
    I_*&=(1-r')(1-R')E.
\end{aligned}
\end{equation*}
Finally, the post-collisional energies can be given in terms of the pre-collisional energies by the following relation
\begin{equation*}
    \begin{aligned}
        I'=&\frac{r(1-R)}{r'(1-R')}I
\\
I'_*=&\frac{(1-r)(1-R)}{(1-r')(1-R')}I_*.
    \end{aligned}
\end{equation*}
The Boltzmann equation for an interacting single polyatomic gas  reads
 \begin{equation}\label{boltzmannequation}
\partial_t{f}+v.\nabla_x{f}=Q(f,f),
 \end{equation}
 where $f=f(t,x,v,I)\geq{0}$ is the distribution function, with $t\geq{0}, x\in\mathbb{R}^3,  v\in\mathbb{R}^3,$ and $I\geq{0}$. The operator $Q(f,f)$ is the quadratic Boltzmann operator \cite{9} given as
 
 \begin{equation}\label{qff}
\begin{aligned}
Q(f, f)(v, I)=\int_{ (0,1)^2 \times S^2\times\mathbb{R}_+\times\mathbb{R}^{3} }\left(\frac{f^{\prime} f_{*}^{\prime}}{\left(I^{\prime} I_{*}^{\prime}\right)^{\alpha}}-\frac{f f_{*}}{\left(I I_{*}\right)^{\alpha}}\right)&\times\mathcal{B} \times (r(1-r))^{\alpha}(1-R)^{2\alpha}\times \\
&\hspace{-0.8cm}I^{\alpha} I_{*}^{\alpha}(1-R) R^{1 / 2} \mathrm{~d} R \mathrm{d} r \mathrm{d} \omega \mathrm{d} I_{*} \mathrm{d} v_{*},
\end{aligned}
 \end{equation}
where $\alpha\geq{0}$, and we use the standard notations $f_*=f(v_*,I_*), f'=f(v',I'),$ and $f'_*=f(v'_*,I'_*)$. {Choosing the power $\alpha$ in the measure of the above integral is essential for the operator $\kay_2$ \eqref{k2} to be a Hilbert-Schmidt operator (see the proof).}
\vspace{1mm}

\noindent The function
 $\mathcal{B}$ is the collision cross-section; a function of $(v,v_*,I,I_*,r,R,\omega)$. In the following, we give some assumptions on $\mathcal{B}$, extended from Grad's assumption for collision kernels of monoatomic gases.
 In general, $\bb$ is assumed to be an almost everywhere positive function satisfying the following microreversibility conditions:
 \begin{equation}\label{reversibility}
 \begin{aligned}
     \bb(v,v_*,I,I_*,r,R,\omega)&=\bb(v_*,v,I_*,I,1-r,R,-\omega),\\
     \bb(v,v_*,I,I_*,r,R,\omega)&=\bb(v',v'_*,I',I'_*,r',R',\omega).
     \end{aligned}
 \end{equation}
 \subsection{Main Assumptions on the Collision Cross-section $\mathcal{B}$}
 \vspace{1mm}
 
 \noindent Throughout this paper, $c>0$ will denote a generic constant. Together with the above assumption \eqref{reversibility}, we assume the following boundedness assumptions on the collision cross section $\mathcal{B}$. In particular, we assume that
 \begin{gather}\label{lbofB}
   \;{\Phi_{\gamma}}(r,R)\;\bigg|\omega.{\frac{(v-v_*)}{|v-v_*|}}\bigg|\Big({|v-v_*|}^{{\gamma}}+I^{\frac{\gamma}{2}}+I_*^{\frac{\gamma}{2}}\Big)\leq \bb(v,v_*,I,I_*,r,R,\omega),
 \end{gather}
 and
 \begin{gather}\label{boundednessofB}
   \bb(v,v_*,I,I_*,r,R,\omega)\leq \;\bigg|\omega.{\frac{(v-v_*)}{|v-v_*|}}\bigg|\Psi_{\gamma}(r,R)\Big({|v-v_*|}^{{\gamma}}+I^{\frac{\gamma}{2}}+I_*^{\frac{\gamma}{2}}\Big), 
 \end{gather}
where $\gamma\geq 0$. In addition, ${\Phi}_{\gamma}$ and  $\Psi_{\gamma}$ are positive functions such that
\begin{equation*}
{\Phi}_{\gamma}\leq \Psi_{\gamma},
\end{equation*}
and
\begin{equation}\label{symmetry}
    \Phi_{\gamma}(r,R)=\Phi_{\gamma}(1-r,R), \quad \Psi_{\gamma}(r,R)=\Psi_{\gamma}(1-r,R).
\end{equation}
In addition, $\Psi_{\gamma}$ satisfies the following 

\begin{equation}\label{k2condition1}
  \Psi_{\gamma}^2(r,R)(r(1-r))^{\min\{2\alpha-1-\gamma,\alpha-1\}}R(1-R)^{3\alpha-\gamma} \in L^1((0,1)^2).
\end{equation}

\noindent In fact, though assumption \eqref{k2condition1} seems to be strict, yet it covers several physical models. In addition, one may notice that as the value of $\alpha$ increases, condition \eqref{k2condition1} covers a wider class of functions $\Psi_{\gamma}$. We give now models for $\mathcal{B}$ that satisfy condition \eqref{k2condition1}.

\subsection{{Models of the Cross-section  $\mathcal{B}$ }}

We present first the relation between the number of atoms in a molecule and the value of $\alpha$, see \cite{physics}.  Let $D$ be the number of degrees of freedom in a molecule of $N$ atoms, then $\alpha$ is given in terms of $D$ in the formula:
\begin{equation}
    \alpha=\frac{D-5}{2}.
\end{equation}
To relate $D$ with $N$, We consider the following cases of molecules:
\begin{enumerate}
    \item \textbf{Non-vibrating molecules}
    In this scenario, we distinguish between linear and non-linear molecules. Regarding the fact that vibrations are not occurring, a linear (respectively non-linear) molecule will always remain linear (respectively non-linear) even after collisions. The number of degrees of freedom $D$ in this case will be the sum of the rotational and translational degrees of freedom in $\mathbb{R}^3,$ and will not depend on the number of atoms in the gas molecule.
    \begin{itemize}
        \item Linear molecules:\begin{itemize}[label={--}]
\item translational degrees of freedom: 3
\item   rotational degrees of freedom: 2
\end{itemize} 
and therefore $D=5$ and $\alpha=0$.
 \item Non-linear molecules:\begin{itemize}[label={--}]
\item translational degrees of freedom: 3
\item   rotational degrees of freedom: 3
\end{itemize} 
and therefore $D=6$ and $\alpha=1/2$. 
    \end{itemize}
    
    \item \textbf{Vibrating molecules} In this case, the total number of degrees of freedom depends on $N$. Consider a molecule of $N$ atoms in $\mathbb{R}^3$, then as long as the shape of the molecule is deformable due to vibrations, the position of each atom will be determined freely by 3 degrees of freedom. Thus, the total number of degrees of freedom of the molecule of $N$ atoms will equal to $3N$. Hence, $D=3N$ and 
    \begin{equation}
    \alpha=\frac{3N-5}{2}.
\end{equation}
 We notice that in this case, the fact that the molecule is linear or non-linear doesn't have an impact on $D$ and $\alpha$.
\end{enumerate}
We give now some models that satisfy assumptions \eqref{k2condition1}, which have been recently studied in the literature (see \cite{milana,doric, Niclas}).\bigskip

\noindent \underline{\textit{Examples} } Consider the case of molecules where vibrations are taken into account, and suppose that 

\noindent $\alpha>\frac{\gamma}{2}$, that is the case of all vibrating polyatomic molecules and the non-vibrating non-linear molecules (where $\alpha=1/2$) if $\gamma\in[0,1)$. Then the following collision cross sections suggested in \cite{milana}
\begin{equation}\label{model3}
\mathcal{B}\left(v, v_{*}, I, I_{*}, r, R,\omega\right)=c\Big|\omega.\frac{v-v_*}{|v-v_*|}\Big|(|v-v_*|^{\gamma}+I^{\gamma/2}+I_*^{\gamma/2}),\end{equation}
{which is identical to the model
\begin{equation}
\mathcal{B}\left(v, v_{*}, I, I_{*}, r, R,\omega\right)=c\Big|\omega.\frac{v-v_*}{|v-v_*|}\Big|E^{\gamma},\end{equation}}
together with the model
\begin{equation}\label{model1}
\begin{aligned}
\mathcal{B}\left(v, v_{*}, I, I_{*}, r, R,\omega\right)=c\Big|\omega.\frac{v-v_*}{|v-v_*|}\Big|\left(R^{\frac{\gamma}{2}}\left|v-v_{*}\right|^{\gamma}\right.
\left.+\left(r(1-R) {I}\right)^{\frac{\gamma}{2}}+\left((1-r)(1-R) {I_{*}}\right)^{\frac{\gamma}{2}}\right)
\end{aligned}
\end{equation}
satisfy \eqref{boundednessofB} by taking for model \eqref{model3}
$$\Phi_{\gamma}(r,R)=\Psi_{\gamma}(r,R)= \Psi_{\gamma}(r,R)=1,$$
and for model \eqref{model1}
$$\Phi_{\gamma}(r,R)=\min\{R,(1-R)\}^{\frac{\gamma}{2}}\min\{r,(1-r)\}^{\frac{\gamma}{2}}, \text{and \;} \Psi_{\gamma}(r,R)=\max\{R^{\frac{\gamma}{2}},(r(1-R) )^{\frac{\gamma}{2}}\}.$$
In \cite{doric}, the authors considered the class of collision cross-sections having the expression
\begin{equation}\label{model2}
\begin{aligned}
\mathcal{B}\left(v, v_{*}, I, I_{*}, r, R,\omega\right)=b\left(\omega.\frac{v-v_*}{|v-v_*|}\right)\left(R^{\frac{\gamma}{2}}\left|v-v_{*}\right|^{\gamma}\right.
\left.+\left(r(1-R) {I}\right)^{\frac{\gamma}{2}}+\left((1-r)(1-R) {I_{*}}\right)^{\frac{\gamma}{2}}\right),
\end{aligned}
\end{equation}
where $b$ was assumed to be $L^1$ integrable on $S^2$ while establishing the first six fields
equations, whereas $b$ was assumed constant for the fourteen moments model.
\bigskip

\subsection{Equivalent Formulation of the Collision Operator}
In this section, we aim to write the collision operator \eqref{qff} in an equivalent form, {which shall be of a great interest in further works. However, this parametrization will not be used in this paper.} The derivation of the latter is a result of subsequent changes of variables, see \eqref{transformation}.  For simplification purpose, we write first the expression of the collision operator in the $\sigma$-notation, where $\sigma=T_{\omega}\Big(\frac{v-v_*}{|v-v_*|}\Big)$. The Jacobian of this transformation (see \cite{us}) is given as
\begin{equation}\label{os}
d\omega=\frac{d\sigma}{2\big|\sigma-\frac{v-v_*}{|v-v_*|}\big|}.
\end{equation}
The final result sought is the Jacobian of the following map:
\begin{equation}\label{transformation}
\begin{aligned}
\mathcal{T}: \mathbb{R}^{6} \times \mathbb{R}^2_{+} \times(0,1)^2\times S^2 & \rightarrow \mathbb{R}^{6} \times \mathbb{R}_{+}^{2}\times  \mathbb{R}^{3} \times \mathbb{R}_{+} \\
\left({v},{v}_{*}, I, I_{*}, r, R,\sigma\right) & \mapsto\left(v, G, E,I, v^{\prime},I^{\prime}\right),
\end{aligned}
\end{equation}
where $g=v-v_*$ and $G=\frac{v+v_*}{2}$. For this transformation, the following Jacobians are elementary:
\begin{equation*}\label{j1}
    J_{(v,v_*I,I_*,r,R,\sigma)\mapsto(g,G,I,I_*,r,R,\sigma)}=1,
\end{equation*}
and 
\begin{equation}\label{j2}
    J_{(g,G,I,I_*,r,R,\sigma)\mapsto(g,G,I,E,r,R,\sigma)}=1.
\end{equation}
Equation \eqref{j2} is due to the fact that only $E$ is a function of $I_*$. What remains in deducing the Jacobian of $\mathcal{T}$ is calculating the Jacobian of the transformation $(g,G,I,E,r,R,\sigma)\mapsto(v,G,I,E,v',I')$. As an intermediate step we define
\begin{equation*}
    \lambda=\sqrt{RE}, \quad \mu=r(1-R),
\end{equation*}
which induces the Jacobian $$J_{(g,G,I,E,r,R,\sigma)\mapsto(g,G,I,E,\lambda,\mu,\sigma)}=\frac{1}{2}\frac{(1-R)}{\sqrt{R}}\sqrt{E}.$$
Thus the final sub-transformation is $(g,G,I,E,\lambda,\mu,\sigma)\mapsto(v,G,I,E,v',I')$, where specifically,
\begin{equation}
    v'= G+\lambda \sigma, \quad \text{and} \quad I'=\mu E.
\end{equation}
It's clear that 
\begin{equation*}
  J_{(g,G,I,E,\lambda,\mu,\sigma)\mapsto(g,G,I,E,\lambda,I',\sigma)}=E
\end{equation*}
and for $v'$ we have
\begin{equation}
J_{(g,G,I,E,\lambda,I',\sigma)\mapsto(g,G,I,E,v',I')}
   ={\lambda}^2 ={RE},
\end{equation}
since $(\lambda,\sigma)$ is the spherical representation of $v'-G$. 
As $v=\frac{1}{2}g+G$, then the Jacobian
\begin{equation}
J_{(g,G,I,E,v',I')\mapsto(v,G,I,E,v',I')}
   =\frac{1}{8}.
\end{equation}
Finally, combining the preceding transformations, the Jacobian of $\mathcal{T}$ is
\begin{equation}
    J_{\mathcal{T}}=\frac{1}{16}R^{\half}(1-R)E^{\frac{5}{2}}.
\end{equation}
In other words,
\begin{equation*}
    \text{d}v  \text{d} G  \text{d}I \text{d}E \text{d}v' \text{d}I'=\frac{1}{16}R^{\half}(1-R)E^{\frac{5}{2}}\text{d}v  \text{d} v_*  \text{d}I \text{d}I_* \text{d}r \text{d}R \text{d}\sigma .
\end{equation*}

The equivalent model of \eqref{qff}, based on the above computations is therefore

\begin{equation}\label{qff3}
\begin{aligned}
Q(f, f)(v, I)=\int_{\mathbb{R}^{3} \times \mathbb{R}_{+}\times  \mathbb{R}^{3} \times \mathbb{R}_{+}}\left(\frac{f^{\prime} f_{*}^{\prime}}{\left(I^{\prime} I_{*}^{\prime}\right)^{\alpha}}-\frac{f f_{*}}{\left(I I_{*}\right)^{\alpha}}\right) W(v,I,v',I',G,E) \;\mathrm{d} G \mathrm{d} E\mathrm{~d} v' \mathrm{d} I',
\end{aligned}
 \end{equation}
 where
 \begin{equation}\label{W}
     W(v,I,v',I',G,E)=\frac{8}{\big|\sigma-\frac{g}{|g|}\big|}\times (I^{\prime}I_*^{\prime }II_*)^{\alpha}\times E^{-\frac{5}{2}-2\alpha}\times\mathcal{B}(v,v_*,   I,I_*,r,R,\sigma),
 \end{equation}
 where $I_*=I_*(v,I,G,E)$, $I'_*=I'_*(v',I',G,E)$, $v'_*=v'_*(G,v')$, $v_*=v_*(G,v)$, $\sigma=\sigma(v',G)$, $R=R(v',E,G)$, and $r=r(I',v',E,G)$.
 \vspace{1mm}
 
\noindent Moreover, $W$ in \eqref{W} is clearly microreversible, 
and the measure $\text{d}E\text{d}G\text{d}v\text{d}I\text{d}v'\text{d}I'$ is obviously invariant. We remark that $\omega$ and $-\omega$ refer to the same $\sigma$, and thus the Jacobian \eqref{os} has to be taken twice in \eqref{qff3}. 

\section{The Linearized Boltzmann Operator}\label{Sec.3}
We state first the H-theorem for polyatomic gases which was initially established in \cite{9}. 
  In particular, the entropy production functional 
 \begin{equation*}
     D(f)=\int_{\mathbb{R}^3}  \int_{\mathbb{R}_+}  Q(f,f)\log f\;\;\; \mathrm{d}I\mathrm{d}v\leq{0},
 \end{equation*}
 and the following are equivalent
 \begin{enumerate}
     \item The collision operator $Q(f,f)$ vanishes, i.e. $Q(f,f)(v,I)=0$ for every $v\in\mathbb{R}^3$ and $I\geq{0}$.
     \item The entropy production vanishes, i.e.
     $D(f)=0.$
     \item There exists $T>0$, $n>0$, and $u\in\mathbb{R}^3$ such that 
     \begin{equation}\label{max1}
         f(v,I)=\frac{n}{(2\pi)^{\frac{3}{2}}\Gamma(\alpha+1)( \kappa T)^{\alpha+\frac{5}{2}}}I^{\alpha}e^{-\frac{1}{kT}(\frac{1}{2}(v-u)^2+I)},
     \end{equation}
 \end{enumerate}
where $\kappa$ in \eqref{max1} is the Boltzmann constant. The linearization of the Boltzmann equation of polyatomic gases is taken around the local Maxwellian function, which  represents the equilibrium state of a gas and is denoted by $M_{n,u,T}(v,I)$, and given by 
\begin{equation}\label{maxwellian1}
        {M}_{n,u,T}(v,I)=\frac{n}{(2\pi)^{\frac{3}{2}}\Gamma(\alpha+1)( \kappa T)^{\alpha+\frac{5}{2}}}I^{\alpha}e^{-\frac{1}{\kappa T}(\frac{1}{2}(v-u)^2+I)},
     \end{equation}
 where $n, u,$ and $T$ in \eqref{maxwellian1} are the number of molecules per unit volume, the hydrodynamic velocity, and the temperature respectively. In particular,
 $$n= \int_{\mathbb{R}^3} \int_{\mathbb{R}_+}  \!\!\!\!f\mathrm{d}I\mathrm{d}v, \quad n u = \int_{\mathbb{R}^3} \int_{\mathbb{R}_+} v f \mathrm{d}I\mathrm{d}v, \quad  {\left(\alpha+\frac{5}{2}\right)n\kappa T}= \int_{\mathbb{R}^3}\int_{\mathbb{R}_+} \bigg(\frac{(v-u)^2}{2} +I\bigg)   f \mathrm{d}I\mathrm{d}v. $$
Without loss of generality, we will consider in the sequel a normalized version ${M}_{1,0,1}$, by assuming ${\kappa T=n=1}$ and $u=0$. For the sake of simplicity, the index will be dropped. In particular,
\begin{equation}\label{eq12}
    {M}(v,I)={M}_{1,0,1}(v,I)=\frac{1}{(2\pi)^{\frac{3}{2}}\Gamma(\alpha+1)} I^{\alpha}e^{-\frac{1}{2}v^2-I}.
\end{equation}
We look for a solution ${f}$ around ${M}$ defined in \eqref{eq12} having the form
\begin{equation*}
    {f}(v,I)={M}(v,I)+ {M}^{\frac{1}{2}}(v,I){g}(v,I).
\end{equation*}
 The linearization of the Boltzmann operator \eqref{qff} around ${M}$ leads to introduce the linearized Boltzmann operator ${\mathcal{L}}$ given as
\begin{equation*}
    \begin{aligned}
       {\mathcal{L}}{g}&={M}^{-\frac{1}{2}}[Q({M},{M}^{\half}{g})+Q({M}^{\frac{1}{2}}{g},{M})].
       \end{aligned}
\end{equation*}
In particular, ${\mathcal{L}}$ writes
\begin{equation}
\begin{aligned}
    {\mathcal{L}}{g}= {M}^{-\frac{1}{2}}\int_{\Delta}\bigg[\frac{M^{\prime}M'_*{}^{\half}}{{\sqrt{I'_{*}{}^{{\alpha}}}}I'{}^{\alpha}}\frac{{g}_*^{\prime}}{\sqrt{I'_{*}{}^{\alpha}}}
    &-\frac{{M}{M}_*^{\half}}{I^{\alpha}\sqrt{I_*^{\alpha}}}\frac{{g}_*}{\sqrt{I_*^{\alpha}}}+\frac{{M}'_*{M}'{}^{\half}}{I'_*{}^{\alpha}\sqrt{I'{}^{\alpha}}}\frac{{g}'}{\sqrt{I'{}^{\alpha}}}-\frac{{M}_*{M}^{\frac{1}{2}}}{I_*^{\alpha}\sqrt{I^{\alpha}}}\frac{{g}}{\sqrt{I^{\alpha}}}\bigg]\times\\ &(r(1-r))^{\alpha}(1-R)^{2\alpha+1}  R^{1/2}I^{\alpha} I_{*}^{\alpha}\mathcal{B}\;\;\mathrm{d}r\mathrm{d}R\mathrm{d}\omega\mathrm{d}I_*\mathrm{d}v_*.
    \end{aligned}
\end{equation}
Thanks to the conservation of total energy \eqref{conservation} we have
 $\frac{M}{I^\alpha }\frac{M_*}{ I_*^\alpha}=\frac{M'}{I^{\prime\alpha} }\frac{M'_*}{ I_*^{\prime\alpha}}$, and so $\mathcal{L}$ has the following form:

\begin{equation*}
\begin{aligned}
\mathcal{L}(g)=
&-{I^{-\frac{\alpha}{2}}}\int_{\Delta}\frac{g_{*}}{I_*^{\frac{\alpha}{2}}}\frac{M^{\frac{1}{2}}}{I^{\frac{\alpha}{2}}} \frac{M_*^{\frac{1}{2}}}{I_*^{\frac{\alpha}{2}}}  (r(1-r))^{\alpha}(1-R)^{2\alpha+1}  R^{1/2}I^{\alpha} I_{*}^{\alpha}\mathcal{B}\, \mathrm{d}r\mathrm{d}R\mathrm{d}\omega\mathrm{d}I_*\mathrm{d}v_*\\
&-I^{-{\alpha}}\int_{\Delta} g \frac{M_*}{I_*^{{\alpha}}} \;(r(1-r))^{\alpha}(1-R)^{2\alpha+1}  R^{1/2}I^{\alpha} I_{*}^{\alpha}\mathcal{B} \, \mathrm{d}r\mathrm{d}R\mathrm{d}\omega\mathrm{d}I_*\mathrm{d}v_*\\
&+I^{-\frac{\alpha}{2}}\int_{\Delta}\frac{g^{\prime}_{*}}{I_*^{\prime\frac{\alpha}{2}}}\frac{M_*^{\frac{1}{2}}}{I_*^{\frac{\alpha}{2}}} \frac{M^{\prime\frac{1}{2}}}{I^{\prime\frac{\alpha}{2}}}   (r(1-r))^{\alpha}(1-R)^{2\alpha+1}  R^{1/2}I^{\alpha} I_{*}^{\alpha}\mathcal{B}\, \mathrm{d}r\mathrm{d}R\mathrm{d}\omega\mathrm{d}I_*\mathrm{d}v_*\\
&+\,I^{-\frac{\alpha}{2}}\int_{\Delta} \frac{g^{\prime}}{I^{\prime\frac{\alpha}{2}}}\frac{M_*^{\frac{1}{2}}}{I_*^{\frac{\alpha}{2}}}  \frac{M_*^{\prime\frac{1}{2}}}{I_*^{\prime\frac{\alpha}{2}}} (r(1-r))^{\alpha}(1-R)^{2\alpha+1}  R^{1/2}I^{\alpha} I_{*}^{\alpha}\mathcal{B}\, \mathrm{d}r\mathrm{d}R\mathrm{d}\omega\mathrm{d}I_*\mathrm{d}v_*.
\end{aligned}
\end{equation*}

\noindent  Here, $\Delta$ refers to the open set $(0,1)^2\times{S^2}\times\mathbb{R}_+\times\mathbb{R}^3$. In addition, $\mathcal{L}$ can be written in the form $$\mathcal{L}=\mathcal{K}-\nu\; \text{Id},$$
where 
\begin{equation}\label{formofk}
  \begin{aligned}
 \kay g&=I^{-\frac{\alpha}{2}}\int_{\Delta}\frac{g^{\prime}_{*}}{I_*^{\prime\frac{\alpha}{2}}} \frac{M_*^{\frac{1}{2}}}{I_*^{\frac{\alpha}{2}}} \frac{M^{\prime\frac{1}{2}}}{I^{\prime\frac{\alpha}{2}}}   (r(1-r))^{\alpha}(1-R)^{2\alpha+1}  R^{1/2}I^{\alpha} I_{*}^{\alpha}\mathcal{B}\, \mathrm{d}r\mathrm{d}R\mathrm{d}\omega\mathrm{d}I_*\mathrm{d}v_*\\
&+\,I^{-\frac{\alpha}{2}}\int_{\Delta}\frac{g^{\prime}}{I^{\prime\frac{\alpha}{2}}} \frac{M_*^{\frac{1}{2}}}{I_*^{\frac{\alpha}{2}}}  \frac{M_*^{\prime\frac{1}{2}}}{I_*^{\prime\frac{\alpha}{2}}} (r(1-r))^{\alpha}(1-R)^{2\alpha+1}  R^{1/2}I^{\alpha} I_{*}^{\alpha}\mathcal{B}\, \mathrm{d}r\mathrm{d}R\mathrm{d}\omega\mathrm{d}I_*\mathrm{d}v_*\\
 &-{I^{-\frac{\alpha}{2}}}\int_{\Delta}\frac{g_{*}}{I_*^{\frac{\alpha}{2}}}\frac{M^{\frac{1}{2}}}{I^{\frac{\alpha}{2}}} \frac{M_*^{\frac{1}{2}}}{I_*^{\frac{\alpha}{2}}}  (r(1-r))^{\alpha}(1-R)^{2\alpha+1}  R^{1/2}I^{\alpha} I_{*}^{\alpha}\mathcal{B}\, \mathrm{d}r\mathrm{d}R\mathrm{d}\omega\mathrm{d}I_*\mathrm{d}v_*,
 \end{aligned}
\end{equation}
and 

\begin{equation}\label{collisionfrequency}
\begin{aligned}\hspace{-1cm}
\nu(v,I)=I^{-{\alpha}}\int_{\Delta} g \frac{M_*}{I_*^{{\alpha}}} \;(r(1-r))^{\alpha}(1-R)^{2\alpha+1}  R^{1/2}I^{\alpha} I_{*}^{\alpha}\mathcal{B} \, \mathrm{d}r\mathrm{d}R\mathrm{d}\omega\mathrm{d}I_*\mathrm{d}v_*,
    \end{aligned}
   \end{equation}
which represents the collision frequency.
We write also $\kay$ as $\kay=\mathcal{K}_3+\mathcal{K}_2-\mathcal{K}_1$ with

    \begin{align}\label{k1}
    \begin{split}
    \mathcal{K}_1g={I^{-\frac{\alpha}{2}}}\int_{\Delta}\frac{g_{*}}{I_*^{\frac{\alpha}{2}}}\frac{M^{\frac{1}{2}}}{I^{\frac{\alpha}{2}}} \frac{M_*^{\frac{1}{2}}}{I_*^{\frac{\alpha}{2}}}  (r(1-r))^{\alpha}(1-R)^{2\alpha+1}  R^{1/2}I^{\alpha} I_{*}^{\alpha}\mathcal{B}\, \mathrm{d}r\mathrm{d}R\mathrm{d}\omega\mathrm{d}I_*\mathrm{d}v_*,
    \end{split}
\end{align}

    \begin{align}\label{k2}
  \mathcal{K}_2g=I^{-\frac{\alpha}{2}}\int_{\Delta}\frac{g^{\prime}_{*}}{I_*^{\prime\frac{\alpha}{2}}} \frac{M_*^{\frac{1}{2}}}{I_*^{\frac{\alpha}{2}}} \frac{M^{\prime\frac{1}{2}}}{I^{\prime\frac{\alpha}{2}}}   (r(1-r))^{\alpha}(1-R)^{2\alpha+1}  R^{1/2}I^{\alpha} I_{*}^{\alpha}\mathcal{B}\, \mathrm{d}r\mathrm{d}R\mathrm{d}\omega\mathrm{d}I_*\mathrm{d}v_*,
  \end{align}
  and
  \begin{equation}\label{k3}
   \kay_3g= I^{-\frac{\alpha}{2}}\int_{\Delta}\frac{g^{\prime}}{I^{\prime\frac{\alpha}{2}}} \frac{M_*^{\frac{1}{2}}}{I_*^{\frac{\alpha}{2}}}  \frac{M_*^{\prime\frac{1}{2}}}{I_*^{\prime\frac{\alpha}{2}}} (r(1-r))^{\alpha}(1-R)^{2\alpha+1}  R^{1/2}I^{\alpha} I_{*}^{\alpha}\mathcal{B}\, \mathrm{d}r\mathrm{d}R\mathrm{d}\omega\mathrm{d}I_*\mathrm{d}v_*.
\end{equation}
  The linearized operator $\mathcal{L}$ is a symmetric operator, with kernel 
  \begin{equation*}
      \text{ker}\;\mathcal{L}={M^{1/2}}\text{ span}\{1,v_i, \frac{1}{2}v^2+I\}\quad i=1,\cdots,3.
  \end{equation*}
  Since $\mathcal{L}$ is symmetric and $\nu$ Id is self-adjoint on $$Dom(\nu\text{ Id})=\{g\in L^2(\mathbb{R}^3\times\mathbb{R}_+): \nu g\in L^2(\mathbb{R}^3\times\mathbb{R}_+) \},$$ then $\kay$ is symmetric.
In the following section, we prove that $\kay$ is a bounded compact operator on $L^2(\mathbb{R}^3\times\mathbb{R}_+)$. Hence, $\mathcal{L}$ is a self adjoint operator on Dom $(\mathcal{L})=\text{Dom}(\nu\text{ Id})$. In section \ref{sec5} we prove that $\nu$ is coercive, and therefore  $\mathcal{L}$ is a Fredholm operator on $L^2(\mathbb{R}^3\times\mathbb{R}_+)$.
\section{Main Result}\label{Sec.4}
We give the main result on the linearized Boltzmann operator based on the assumptions of the collision cross section \eqref{boundednessofB} and \eqref{lbofB}. 
In particular, using \eqref{lbofB} we prove that the multiplication operator by $\nu$ Id is coercive and using \eqref{boundednessofB} we prove that $\mathcal{K}$ is compact. This leads to the Fredholm property of $\mathcal{L}$ on $L^2(\mathbb{R}^3\times\mathbb{R}_+)$.
\bigskip

\noindent We give first the following lemma, which will be used in the proof of our result.

\begin{lemma}\label{lemma1}
There exists $0<c_1\leq c_2$ such that for any $a,b\in \mathbb{R}$ the following inequality holds
\begin{equation}
    c_1\leq \frac{\max \{r^{a}(1-r)^{b},(1-r)^{a}r^{b}\}}{(r(1-r))^{\min\{a, b\}}}\leq c_2
\end{equation}
where $r\in [0,1].$
\end{lemma}

 \noindent We state now the following theorem, which is the main result of the paper.
\begin{theorem}\label{prop2}
The operator $\kay$ of polyatomic gases defined in \eqref{formofk} is a compact operator from $L^2(\mathbb{R}^3\times\mathbb{R}_+)$ to $L^2(\mathbb{R}^3\times\mathbb{R}_+)$, and the multiplication operator by $\nu$ is coercive. As a result, the linearized Boltzmann operator $\mathcal{L}$ is an unbounded self adjoint Fredholm operator from Dom$(\mathcal{L})=$Dom$(\nu \text{ Id})\subset L^2(\mathbb{R}^3\times\mathbb{R}_+)$ to $L^2(\mathbb{R}^3\times\mathbb{R}_+)$.
\end{theorem}
\vspace{1mm}

\noindent We carry out the proof of the coercivity of $\nu$ Id in section \ref{sec5}, and we dedicate the rest of this section for the proof of the compactness of $\kay$. 
 
\begin{proof}
Throughout the proof, we prove the compactness of each $\kay_i$ with $i=1,\cdots,3$ separately.\vspace{2mm}

\noindent\textit{Compactness of $\kay_1$.} The compactness of $\kay_1$ is straightforward as $\kay_1$ already possesses a kernel form. Thus, we can inspect the operator kernel of $\kay_1$  $\eqref{k1}$  to be
\begin{equation*}
    k_1(v,I,v_*,I_*)=\frac{ 1}{\Gamma(\alpha+1)(2\pi)^{\frac{3}{2}}}\int_{(0,1)^2\times S^2}(r(1-r))^{\alpha}(1-R)^{2\alpha+1}  R^{1/2} I^{\frac{\alpha}{2}}I_{*}^{\frac{\alpha}{2}}\mathcal{B}e^{-\frac{1}{4}v^2_*-\frac{1}{4}v^2-\half I_*-\half I}\,\mathrm{d}r\mathrm{d}R\mathrm{d}\omega,
\end{equation*}
and therefore
\begin{equation*}
    \mathcal{K}_{1} g(v,I)=\int_{\mathbb{R}^{3}\times\mathbb{R}_+} g\left(v_{*},I_*\right) k_{1}\left(v, I,v_{*},I_*\right) \text{d}v_*\text{d}I_* \quad \forall (v,I) \in \mathbb{R}^{3}\times\mathbb{R}_+.
\end{equation*}
\newcommand{\real}{\mathbb{R}}
We give the following lemma that yields to the compactness of $\kay_1$. 
\begin{lemma}
With the assumption \eqref{boundednessofB} on $\mathcal{B}$, the function $k_1\in{L^2({\real^3\times\real_+\times\real^3\times\real_+})}.$ 
\end{lemma}

\begin{proof}
Applying Cauchy-Schwarz we get
\begin{equation*}
    \begin{aligned}
       ||{k_1}||^2_{L^2}&\leq c\int_{\real^3}\int_{\real_+}\int_{\real^3}\int_{\real_+} I^{\alpha}I_*^{\alpha}(I^{{\gamma}}+I_*^{{\gamma}}+|v-v_*|^{2\gamma}) e^{-\frac{1}{2}v^2_*-\frac{1}{2}v^2- I_*-I}\text{d}I\text{d}v\text{d}I_*\text{d}v_*\\
       &\leq c\int_{\real^3} e^{-\frac{1}{2}v^2_*}\bigg[\int_{|v-v_*|\leq 1} e^{-\frac{1}{2}v^2}\mathrm{d}v+\int_{|v-v_*|\geq 1} |v-v_*|^{\lceil 2\gamma \rceil }e^{-\frac{1}{2}v^2}\mathrm{d}v\bigg]\mathrm{d}v_*\\
      & \leq c \int_{\real^3} e^{-\frac{1}{2}v^2_*}\left[\int_{|v-v_*|\geq 1}\sum_{k=0}^{\lceil 2\gamma \rceil}|v|^{k}|v_*|^{\lceil 2\gamma \rceil-k} e^{-\frac{1}{2}v^2}\mathrm{d}v\right] \mathrm{d}v_*\\
    &\leq c\sum_{k=0}^{\lceil 2\gamma \rceil}\int_{\real^3}|v_*|^{\lceil 2\gamma \rceil-k} e^{-\frac{1}{2}v^2_*}\left[\int_{\real^3}|v|^{k} e^{-\frac{1}{2}v^2}\mathrm{d}v \right]\mathrm{d}v_*
    <\infty,
    \end{aligned}
\end{equation*}
where $\lceil 2\gamma \rceil$ is the ceiling of $2\gamma$.
\end{proof}
\vspace{0.1mm}

\noindent This implies that $\kay_1$ is a Hilbert-Schmidt operator, and thus compact. We prove now the compactness of $\kay_2$, similarly by proving it to be a 
Hilbert-Schmidt Operator.
\vspace{3mm}

\noindent\textit{Compactness of $\kay_2$.}   Additional work is required to inspect the kernel form of $\kay_2$, since the kernel is not obvious.  As a first step, we simplify the expression of $\kay_2$ by writing it in the $\sigma-$notation and taking the Jacobian \eqref{os} into consideration. We thus write $\kay_{2}$ as
\begin{equation}\label{kay2}
    \begin{aligned}
   \kay_{2}g(v,I)= \int_{\Delta}&I^{\prime -\frac{\alpha}{2}}_*e^{-\frac{I_*}{2}-\frac{1}{2}r(1-R)\big(\frac{(v-v_*)^2}{4}+I+I_*\big)-\frac{1}{4}v_*^2-\frac{1}{4}\big(\frac{v+v_*}{2}+\sqrt{{R(\frac{1}{4}(v-v_*)^2+I+I_*)}}\,\sigma\big)^2}\times\\
    &\hspace{-1cm}g\left({\frac{v+v_*}{2}-\sqrt{{R(\frac{1}{4}(v-v_*)^2+I+I_*)}}}\sigma,(1-R)(1-r)\Big[\frac{1}{4}(v-v_*)^2+I+I_*\Big]\right)\\
    &\hspace{0.5cm}\frac{1}{\Gamma(\alpha+1)(2\pi)^{\frac{3}{2}}}(r(1-r))^{\alpha}(1-R)^{2\alpha+1}  R^{1/2}I^{\frac{\alpha}{2}} I_{*}^{\alpha}\mathcal{B}\,\,{{\bigg|\sigma-\frac{v-v_*}{|v-v_*|}\bigg|^{-1}}}\mathrm{d}r\mathrm{d}R\mathrm{d}\sigma\mathrm{d}I_*\mathrm{d}v_*.
    \end{aligned}
\end{equation}

\noindent We seek then to write $\kay_2$ in its kernel form. For this, we define $h_{v,I,r,R,\sigma}$; where for simplicity the index will be omitted; as
\begin{equation*}
\begin{aligned}
    h: \real^3\times\real_{+}&\longmapsto h(\real^3\times\real_{+})\subset{\real^3\times\real_{+}}\\
    (v_*,I_*)\;\;&\longmapsto(x,y)=\Big({\frac{v+v_*}{2}-\sqrt{{R(\frac{1}{4}(v-v_*)^2+I+I_*)}}}\sigma,\\
    &\hspace{5.5cm}(1-R)(1-r)\Big[\frac{1}{4}(v-v_*)^2+I+I_*\Big]\Big),
\end{aligned}
\end{equation*} 
for fixed $v$,$I$,$r$,$R$, and $\sigma$. The function $h$ is invertible,  and $(v_*,I_*,v',I')$ can be expressed in terms of $(x,y)$ as
\begin{equation*}
    \begin{aligned}
        v_*&=2x+2\sqrt{Ray}\sigma-v,
       \quad  I_*&=ay-I-(x-v+\sqrt{Ray}\sigma)^2,
    \end{aligned}
\end{equation*}
 and
\begin{equation*}
    \begin{aligned}
        v'&=x+2\sqrt{Ray}\sigma,
      \quad
        I'&=\frac{r}{1-r}y,
    \end{aligned}
\end{equation*}
where $a=\frac{1}{(1-r)(1-R)}$.
The Jacobian of $h^{-1}$ is computed as
\begin{equation*}
    J=\left|\frac{\partial v_* \partial I_*}{\partial x \partial y}\right|=\frac{8}{(1-r)(1-R)},
\end{equation*}
and  the positivity of $I_*$ restricts the variation of 
the variables $(x,y)$ in integral \eqref{kay2}  over the space
\begin{equation}\label{hrr}
   H_{R,r,\sigma}^{v,I}=h(\real^3\times\real_+)=\{(x,y)\in\mathbb{R}^3\times\mathbb{R}_+:ay-I-(x-v+\sqrt{Ray}\sigma)^2>{0}\}.
\end{equation}
In fact, $H_{R,r,\sigma}^{v,I}$ can be explicitly expressed as 
\begin{equation*}
H_{R,r,\sigma}^{v,I}=\{(x,y)\in \mathbb{R}^3\times\mathbb{R}_+\;:\; x\in {B_{v-\sqrt{Ray}\sigma}(\sqrt{ay-I})}\text{ and } y\in ((1-r)(1-R)I,+\infty)\}.
\end{equation*}

\noindent Therefore, equation ($\ref{kay2}$) becomes

\begin{gather}\label{secondkay2}
    \kay_{2}g=\frac{1}{\Gamma(\alpha+1)(2\pi)^{\frac{3}{2}}}\int_{(0,1)^2\times S^2}\int_{H_{R,r,\sigma}^{v,I}}y^{ -\frac{\alpha}{2}}(r(1-r))^{\alpha}(1-R)^{2\alpha+1}  R^{1/2}I^{\frac{\alpha}{2}} I_{*}^{\alpha}\mathcal{B}J\Big|\sigma-\frac{v-x-\sqrt{Ray}\sigma}{|v-x-\sqrt{Ray}\sigma|}\Big|^{-1}\\g(x,y)
   e^{-\frac{ay-I-(x-v+\sqrt{Ray}\sigma)^2}{2}-\frac{r}{2(1-r)}y-\frac{1}{4}(2x+2\sqrt{Ray}\sigma-v)^2-\frac{1}{4}(x+2\sqrt{Ray}\sigma)^2} \mathrm{d}y\mathrm{d}x\mathrm{d}\sigma\mathrm{d}r\mathrm{d}R.
   \end{gather}
We now point out the kernel form of $\kay_2$ and prove after by the help of assumption \eqref{boundednessofB} that the kernel of $\kay_2$ is in $L^2(\real^3\!\!\times\!\!\preal\!\!\times\!\!\real^3\!\!\times\!\!\preal)$. Indeed, we recall the definition of $\Delta$, with $\Delta :=  (0,1)^2\times{S^2}\times\mathbb{R}_+\times\mathbb{R}^3),$ and we define $H^{v,I}$ to be
\begin{equation*}
\begin{aligned}
H^{v,I} := \{(R,r,\sigma,x,y)  \in \Delta \;:\; R\in(0,1),\; r\in(0,1),\; \sigma\in S^2,& \; x\in {B_{v-\sqrt{Ray}\sigma}(\sqrt{ay-I})} ,\\ 
&\hspace{-1.5cm}\text{ and } y\in ((1-r)(1-R)I,+\infty)\}.
\end{aligned}
\end{equation*}
We remark that $H_{R,r,\sigma}^{v,I}$ is a slice of $H^{v,I}$, and we define the slice
 $H^{v,I}_{x,y} \subset (0,1) \times (0,1) \times {S}^2$ such that $H^{v,I}=   H^{v,I}_{x,y}\times\mathbb{R}^3 \times \mathbb{R}_+$.
In particular, 
 \begin{equation}
 H^{v,I}_{x,y}=\{(r,R,\sigma)\in (0,1)\times (0,1)\times S^2 : (y,x,\sigma,r,R)\in H^{v,I} \}.
 \end{equation}
 Then by Fubini theorem, it holds that
\begin{equation}
    \begin{aligned}
    \kay_{2}g (v,I)=&\frac{1}{\Gamma(\alpha+1)(2\pi)^{\frac{3}{2}}}\int_{H^{v,I}}y^{ -\frac{\alpha}{2}}(r(1-r))^{\alpha}(1-R)^{2\alpha+1}  R^{1/2}I^{\frac{\alpha}{2}} I_{*}^{\alpha}\mathcal{B}J\Big|\sigma-\frac{v-x-\sqrt{Ray}\sigma}{|v-x-\sqrt{Ray}\sigma|}\Big|^{-1}g(x,y)\times\\
    &e^{-\frac{ay-I-(x-v+\sqrt{Ray}\sigma)^2}{2}-\frac{r}{2(1-r)}y-\frac{1}{4}(2x+2\sqrt{Ray}\sigma-v)^2-\frac{1}{4}(x+2\sqrt{Ray}\sigma)^2} \mathrm{d}r\mathrm{d}R\mathrm{d}\sigma\mathrm{d}x\mathrm{d}y\\
&= \frac{1}{\Gamma(\alpha+1)(2\pi)^{\frac{3}{2}}}\int_{\real^3\times\real_+}\int_{H^{v,I}_{x,y}} y^{ -\frac{\alpha}{2}}(r(1-r))^{\alpha}(1-R)^{2\alpha+1}  R^{1/2}I^{\frac{\alpha}{2}} I_{*}^{\alpha}\mathcal{B}J\Big|\sigma-\frac{v-x-\sqrt{Ray}\sigma}{|v-x-\sqrt{Ray}\sigma|}\Big|^{-1}\\
&\hspace{5mm}g(x,y)e^{-\frac{ay-I-(x-v+\sqrt{Ray}\sigma)^2}{2}-\frac{r}{2(1-r)}y-\frac{1}{4}(2x+2\sqrt{Ray}\sigma-v)^2-\frac{1}{4}(x+2\sqrt{Ray}\sigma)^2} \mathrm{d}r\mathrm{d}R\mathrm{d}\sigma \mathrm{d}y\mathrm{d}x.
   \end{aligned}
\end{equation}
The kernel of $\kay_2$ is thus inspected and written explicitly in the following lemma.


\begin{lemma}
With the assumption \eqref{boundednessofB} on $\mathcal{B}$, the kernel of $\kay_2$ given by 
\begin{equation*}
\begin{aligned}
    k_2(v,I,x,y)=&\frac{1}{\Gamma(\alpha+1)(2\pi)^{\frac{3}{2}}}\int_{H^{v,I}_{x,y}}y^{-\frac{\alpha}{2}}(r(1-r))^{\alpha}(1-R)^{2\alpha+1}  R^{1/2}I^{\frac{\alpha}{2}} I_{*}^{\alpha}\mathcal{B}J\left|\sigma-\frac{v-x-\sqrt{Ray}\sigma}{|v-x-\sqrt{Ray}\sigma|}\right|^{-1}\times \\
   &e^{-\frac{ay-I-(x-v+\sqrt{Ray}\sigma)^2}{2}-\frac{r}{2(1-r)}y-\frac{1}{4}(2x+2\sqrt{Ray}\sigma-v)^2-\frac{1}{4}(x+2\sqrt{Ray}\sigma)^2} \mathrm{d}r\mathrm{d}R
  \mathrm{d}\sigma  \end{aligned}
\end{equation*}

\noindent is in $L^2(\;\real^3\!\!\times\!\!\preal\!\!\times\!\! \real^3\!\!\times\!\!\preal )$. 
\end{lemma}

\begin{proof}
Applying Cauchy-Schwarz inequality we get 
\begin{equation*}
\begin{aligned}
    \left\lVert{k_2}\right\rVert^2_{L^2}\leq c\int_{\real^3}&\int_{\preal}\int_{\real^3}\int_{\preal}\int_{(0,1)^2\times S^2}y^{-\alpha}(r(1-r))^{2\alpha}(1-R)^{4\alpha+2}  RI^{\alpha} I_{*}^{2\alpha}J^2\mathcal{B}^2\left|\sigma-\frac{v-x-\sqrt{Ray}\sigma}{|v-x-\sqrt{Ray}\sigma|}\right|^{-2}\\
    &\hspace{1cm} e^{-[ay-I-(x-v+\sqrt{Ray}\sigma)^2]-\frac{r}{(1-r)}y-\frac{1}{2}(2x+2\sqrt{Ray}\sigma-v)^2-\frac{1}{2}(x+2\sqrt{Ray}\sigma)^2}\mathrm{d}r\mathrm{d}R\mathrm{d}\sigma\mathrm{d}y\mathrm{d}x\mathrm{d}I\mathrm{d}v.
    \end{aligned}
 \end{equation*}
 By means of $h^{-1}$ we have then
 
\begin{equation}\label{normofk2}
\begin{aligned}
    \left\lVert{k_2}\right\rVert^2_{L^2}\leq c\int_{\real^3}\int_{\preal}\int_{\real^3}\int_{\preal}\int_{(0,1)^2\times S^2}&E^{-\alpha}e^{-{I_*}-\frac{1}{2}v_*^2-r(1-R)\big(\frac{(v-v_*)^2}{4}+I+I_*\big)-\frac{1}{2}\big(\frac{v+v_*}{2}+\sqrt{{R(\frac{1}{4}(v-v_*)^2+I+I_*})}\sigma\big)^2}
    \\&\hspace{-1.9cm}r^{2\alpha}(1-r)^{{\alpha}}(1-R)^{3\alpha+2}  RI^{\alpha} I_{*}^{2\alpha}J\mathcal{B}^2\left|\sigma-\frac{v-x-\sqrt{Ray}\sigma}{|v-x-\sqrt{Ray}\sigma|}\right|^{-2} \mathrm{d}r\mathrm{d}R\mathrm{d}\sigma\mathrm{d}I_*\mathrm{d}v_*\mathrm{d}I\mathrm{d}v.
    \end{aligned}
 \end{equation}
 \noindent Using the inequalities

\begin{equation}\label{inequalities}
    {I}^{\alpha}\leq\left(\frac{1}{4}(v-v_*)^2+I+I_*\right)^{\alpha}=E^{\alpha}\;\text{ and }\;\; (v-v_*)^{2\gamma}+I^{\gamma}+I_*^{\gamma}\leq cE^{\gamma},
\end{equation}
\noindent and using assumption \eqref{boundednessofB} on $\mathcal{B}$ we get
\begin{equation}\label{1}
\begin{aligned}
    \left\lVert{k_2}\right\rVert^2_{L^2}&\leq c\int_{(0,1)^2\times S^2}\int_{\real^3}\int_{\preal}\int_{\real^3}\int_{\preal}e^{-{I_*}-\frac{1}{2}v_*^2-r(1-R)\big(\frac{(v-v_*)^2}{4}+I+I_*\big)-\frac{1}{2}\big(\frac{v+v_*}{2}+\sqrt{{R(\frac{1}{4}(v-v_*)^2+I+I_*})}\sigma\big)^2}\\
    &\hspace{3cm}\Psi^2_{\gamma}(r,R)E^{\gamma}r^{2\alpha}(1-r)^{\alpha}(1-R)^{3\alpha+2} J R I_{*}^{2\alpha} 
  \mathrm{d}I\mathrm{d}v\mathrm{d}I_*\mathrm{d}v_*\mathrm{d}r\mathrm{d}R\mathrm{d}\sigma.
  \end{aligned}
  \end{equation}
{We remark that choosing $\alpha$ to be the power of the measure of integral \eqref{qff}, is essential for eliminating $I^{\alpha}$ from \eqref{normofk2}, which is not integrable. This elimination is possible thanks to the first inequality of \eqref{inequalities}.}
  Perform now the change of variable $I\longmapsto E=I+I_*+\frac{1}{4}|v-v_*|^2$, then as $dI=dE$, \eqref{1} becomes
   \begin{equation}
  \begin{aligned}
   \left\lVert{k_2}\right\rVert^2_{L^2}& \leq c\int_{(0,1)^2\times S^2}\int_{\real^3}\int_{\preal}\int_{\real^3}\int_{\preal} e^{-{I_*}-\frac{1}{2}v_*^2-r(1-R)E-\frac{1}{2}\big(\frac{v+v_*}{2}+\sqrt{{RE}}\sigma\big)^2} \\
    &\hspace{3cm}\Psi^2_{\gamma}(r,R)r^{2\alpha}(1-r)^{\alpha-1}(1-R)^{3\alpha+1}  R I_{*}^{{2\alpha}}E^{\gamma}\mathrm{d}E\mathrm{d}v\mathrm{d}I_*\mathrm{d}v_*\mathrm{d}r\mathrm{d}R\mathrm{d}\sigma\\
    &=c\int_{(0,1)^2}\int_{\real^3}\int_{\preal}\int_{\preal} e^{-{I_*}-\frac{1}{2}v_*^2-r(1-R)E}\left[\int_{S^2}\int_{\real^3}e^{-\frac{1}{2}\big(\frac{v+v_*}{2}+\sqrt{{RE}}\sigma\big)^2}\mathrm{d}v\mathrm{d}\sigma\right] \\
    &\hspace{3cm}\Psi^2_{\gamma}(r,R)r^{2\alpha}(1-r)^{\alpha-1}(1-R)^{3\alpha+1}  R I_{*}^{{2\alpha}}E^{\gamma}\mathrm{d}E\mathrm{d}I_*\mathrm{d}v_*\mathrm{d}r\mathrm{d}R.
    \end{aligned}
 \end{equation}
Let $\Tilde{V}=\frac{v}{2}+\frac{v_*}{2}+\sqrt{{RE}}\sigma$, then 

   \begin{equation}\label{3}
  \begin{aligned}
   \left\lVert{k_2}\right\rVert^2_{L^2}& 
    \leq c\int_{(0,1)^2}\int_{\real^3}\int_{\preal}\int_{\preal} e^{-{I_*}-\frac{1}{2}v_*^2-r(1-R)E}\left[\int_{\real^3}\int_{S^2}e^{-\frac{1}{2}\Tilde{V}^2}\mathrm{d}\Tilde{V}\mathrm{d}\sigma\right]\\
    &\hspace{3cm}\Psi^2_{\gamma}(r,R)r^{2\alpha}(1-r)^{\alpha-1}(1-R)^{3\alpha+1}  R I_{*}^{{2\alpha}}E^{\gamma}\mathrm{d}E\mathrm{d}I_*\mathrm{d}v_*\mathrm{d}r\mathrm{d}R.
    \end{aligned}
 \end{equation}
  Therefore the integral in \eqref{3} becomes
 \begin{equation}
  \begin{aligned}
    \left\lVert{k_2}\right\rVert^2_{L^2}& 
    \leq \int_{(0,1)^2}\left[\int_{\real_+}E^{\gamma}e^{-r(1-R)E}dE\right]\Psi^2_{\gamma}(r,R) r^{2\alpha}(1-r)^{\alpha-1}(1-R)^{3\alpha+1}  R drdR\\
      &\leq c\int_{(0,1)^2}\Psi^2_{\gamma}(r,R)r^{2\alpha-1-\gamma}(1-r)^{\alpha-1}R(1-R)^{3\alpha-\gamma}  \mathrm{d}r\mathrm{d}R<\infty.
 \end{aligned}
 \end{equation}
\end{proof}
\bigskip


\noindent\textit{Compactness of $\kay_3$.} The proof of the compactness of $\kay_3$ \eqref{k3} is similar to that of $\kay_2$. The operator $\kay_3$ which has the explicit form
  \begin{equation*}
    \begin{aligned}
    \kay_{3}g(v,I)=\int_{\Delta}\,\,&e^{-\frac{I_*}{2}-\frac{1}{2}(1-r)(1-R)\big(\frac{(v-v_*)^2}{4}+I+I_*\big)}e^{-\frac{1}{4}v_*^2-\frac{1}{4}\big(\frac{v+v_*}{2}-\sqrt{{R(\frac{1}{4}(v-v_*)^2+I+I_*)}}\,\sigma\big)^2}\\
    & I^{\prime -\alpha}g\Big({\frac{v+v_*}{2}+\sqrt{{R(\frac{1}{4}(v-v_*)^2+I+I_*)}}}\sigma,r(1-R)\Big[\frac{1}{4}(v-v_*)^2+I+I_*\Big]\Big)\\
   &\hspace{0.5cm}\frac{1}{\Gamma(\alpha+1)(2\pi)^{\frac{3}{2}}}(r(1-r))^{\alpha}(1-R)^{2\alpha+1}  R^{1/2}I^{\frac{\alpha}{2}} I_{*}^{\alpha}\mathcal{B}\,\,{{\bigg|\sigma-\frac{v-v_*}{|v-v_*|}\bigg|^{-1}}}\mathrm{d}r\mathrm{d}R\mathrm{d}\sigma\mathrm{d}I_*\mathrm{d}v_*,
    \end{aligned}
\end{equation*}

\noindent inherits the same form as $\kay_2$, with a remark that the Jacobian of the needed transformation 
  \begin{equation*}
\begin{aligned}
   \tilde{ h}: \real^3\times\real_{+}&\longmapsto\real^3\times\real_{+}\\
    (v_*,I_*)&\longmapsto(x,y)=\Big({\frac{v+v_*}{2}+\sqrt{{R(\frac{1}{4}(v-v_*)^2+I+I_*)}}}\sigma,\\
    &\hspace{7cm}r(1-R)\Big[\frac{1}{4}(v-v_*)^2+I+I_*\Big]\Big),
\end{aligned}
\end{equation*} 
    is calculated to be 
    \begin{equation*}
        \tilde{J}=\frac{8}{r(1-R)}.
    \end{equation*}
For the kernel of $\kay_3$ to be $L^2$ integrable, the final computations require 
\begin{equation}\label{k3condition1}
   {\Psi^2_{\gamma}}(r,R)(1-r)^{2\alpha-1-\gamma}r^{\alpha-1}R(1-R)^{3\alpha-\gamma} \quad \in L^1((0,1)^2), 
\end{equation}
As $\Psi_{\gamma}(r,R)=\Psi_{\gamma}(1-r,R)$ by \eqref{symmetry}, then \eqref{k3condition1} is satisfied.
\end{proof}
\noindent  To this extent, {the perturbation operator $\kay$ is proved to be Hilbert-Schmidt}, and thus $\kay$ is a {bounded }compact operator. As a result, the linearized operator $\mathcal{L}$ is a {self adjoint} operator.
\section{Properties of the Collision Frequency}\label{sec5}
We give in this section some properties of $\nu$. The first is the coercivity property, which implies that $\mathcal{L}$ is a Fredholm operator, and we prove the monotony of $\nu$ which depends on the choice of the collision cross section $\mathcal{B}$. The latter property is used for locating the essential spectrum of $\mathcal{L}$.
\begin{proposition}[Coercivity of $\nu$ Id]\label{coercivity}
 With the assumption \eqref{lbofB}, there exists $c>0$ such that
 $$\nu(v,I)\geq c(|v|^{\gamma}+I^{\frac{\gamma}{2}}+1),$$
 for any $\gamma\geq0$.
 As a result, the multiplication operator $\nu$ \emph{Id} is coercive.
 \end{proposition}
 \begin{proof} The collision frequency \eqref{collisionfrequency} is
 \begin{equation*}
\begin{aligned}
{\nu}(v,I)=\frac{1}{\Gamma(\alpha+1)(2\pi)^{\frac{3}{2}}}\int_{\Delta} \mathcal{B} I_*^{\alpha}(r(1-r))^{\alpha}(1-R)^{2\alpha+1}  R^{1/2}  e^{-I_*-\frac{1}{2} v^2_*} \, \mathrm{d}r\mathrm{d}R\mathrm{d}\omega\mathrm{d}I_*\mathrm{d}v_*,
\end{aligned}
\end{equation*}
where by $\eqref{lbofB}$ we get
\begin{equation*}
\begin{aligned}
{\nu}(v,I)&\geq c\int_{\mathbb{R}^3}\left( |v-v_*|^{\gamma}+I^{\gamma/2}\right) e^{-\frac{1}{2} v^2_*} \, \mathrm{d}v_*\\
&\geq c\Big(I^{\gamma/2}+\int_{\mathbb{R}^3} ||v|-|v_*||^{\gamma} e^{-\frac{1}{2} v^2_*} \,\mathrm{d}v_*\Big),
\end{aligned}
\end{equation*}
where $c$ is a generic constant.
We consider the two cases, $|v|\geq 1$ and $|v|\leq 1$.
If $|v|\geq 1$ we have
\begin{equation*}
    \begin{aligned}
    \nu(v,I)&\geq c\Big(I^{\gamma/2}+\int_{|v_*|\leq \frac{1}{2}|v|} (|v|-|v_*|)^{\gamma} e^{-\frac{1}{2} v^2_*} \,\mathrm{d}v_*\Big)
    \geq c\Big(I^{\gamma/2}+|v|^{\gamma}\int_{|v_*|\leq \frac{1}{2}}  e^{-\frac{1}{2} v^2_*}\mathrm{d}v_*\Big)\geq c(|v|^{\gamma}+I^{\gamma/2}+1). 
    \end{aligned}
\end{equation*}
For $|v|\leq 1$,
\begin{equation*}
    \begin{aligned}
    \nu(v,I)&\geq c\Big(I^{\gamma/2}+\int_{|v_*|\geq 2}  (|v_*|-|v|)^{\gamma}e^{-\frac{1}{2} v^2_*} \,\mathrm{d}v_*\Big)
    \geq c\Big(I^{\gamma/2}+|v|^{\gamma}\int_{|v_*|\geq 2}  e^{-\frac{1}{2} v^2_*} \,\mathrm{d}v_*\Big)\geq c(1+I^{\gamma/2}+|v|^{\gamma}).
    \end{aligned}
\end{equation*}
 \end{proof}
The result is thus proved. We give now the following proposition, which is a generalization of the work of Grad \cite{grad}, in which he proved that the collision frequency of monoatomic single gases is monotonic based on the choice of the collision cross section $\bb$.

\begin{proposition}[monotony of $\nu$]
Under the assumption that 
\begin{equation}\label{monotonicity}
   \int_{(0,1)^2\times S^2}(r(1-r))^{\alpha}(1-R)^{2\alpha+1}  R^{1/2} \mathcal{B}(|V|,I,I_*,r,R,\sigma) \mathrm{d}r\mathrm{d}R\mathrm{d}\omega
\end{equation}
 is increasing (respectively decreasing) in $|V|$ and $I$ for every $I_*$,
the collision frequency $\nu$ is increasing (respectively decreasing), where $|V|=|v-v_*|$.
\vspace{2mm}

\noindent We remark that if $\mathcal{B}$ is increasing (respectively decreasing) in $|V|$ and $I$, then \eqref{monotonicity} is increasing (respectively decreasing) in $|V|$ and $I$.
\vspace{2mm}

\noindent In particular, for Maxwell molecules, where $\mathcal{B}$ is constant in $|V|$ and $I$, $\nu$ is constant. On the other hand, for collision cross-sections of the form 
\begin{gather*}\label{vv}
   \bb(v,v_*,I,I_*,r,R,\omega)= \Phi_{\gamma}(r,R)\bigg|\omega.{\frac{(v-v_*)}{|v-v_*|}}\bigg|\Big({|v-v_*|}^{{\gamma}}+I^{\frac{\gamma}{2}}+I_*^{\frac{\gamma}{2}}\Big), 
 \end{gather*}
{ the integral \eqref{monotonicity} is increasing, and thus $\nu$ is increasing}, where $C>0$, $\gamma\geq 0$, and $\Phi_{\gamma}$ is a positive function such that
$$\Phi(r,R)=\Phi(1-r,R),$$
and
$$\Phi_{\gamma}(r,R)(r(1-r))^{\alpha}R^{\frac{1}{2}}(1-R)^{2\alpha+1} \in L^1((0,1)^2).$$ \end{proposition}
\vspace{1mm}

{ \noindent In fact, if $\Phi_{\gamma}$ for instance satisfies
$$ \Phi_{\gamma}^2(r,R)(r(1-r))^{2\alpha-1-\gamma}R(1-R)^{3\alpha-\gamma} \in L^1((0,1)^2)$$
then this collision cross section satisfies our main assumptions \eqref{lbofB} and \eqref{boundednessofB}.}

\begin{proof}
We remark first that $\nu$ is a radial function in $|v|$ and I. In fact, we perform the change of variable $V=v-v_*$ in the integral \eqref{collisionfrequency}, where the expression of $\nu$ becomes 
\begin{equation}\label{nuV}
    \nu(|v|,I)=\frac{1}{\Gamma(\alpha+1)(2\pi)^{\frac{3}{2}}}\int_{\Delta}\mathcal{B}(|V|,I,I_*,r,R,\omega)I_*^{\alpha}(r(1-r))^{\alpha}(1-R)^{2\alpha+1}  R^{1/2}e^{-\frac{1}{2}(v-V)^2-I_*} \mathrm{d}r\mathrm{d}R\mathrm{d}\omega\mathrm{d}I_*\mathrm{d}V,
\end{equation}
where 
$\Delta=\mathbb{R}^3\times\mathbb{R}_+\times{S^2}\times(0,1)^2$.
The integration in $V$ in the above integral \eqref{nuV} is carried in the spherical coordinates of $V$, with fixing one of the axes of the reference frame along $v$, and therefore, the above integral will be a function of $|v|$ and $I$.  
\vspace{1mm}

\noindent The partial derivative of $\nu$ in the $v_i$ direction is

\begin{equation}\label{pdofnu}
    \frac{\partial{\nu}}{\partial{v_i}}=\int \frac{I_*^{\alpha}(r(1-r))^{\alpha}(1-R)^{2\alpha+1}  R^{1/2}}{\Gamma(\alpha+1)(2\pi)^{\frac{3}{2}}}  \frac{v_i-v_{*i}}{|v-v_{*}|}\frac{\partial\mathcal{B}}{\partial|v-v_*|}(|v-v_*|,I,I_*,r,R,\omega)e^{-\frac{1}{2}v^2_*-I_*}\mathrm{d}r\mathrm{d}R\mathrm{d}\omega\mathrm{d}I_*\mathrm{d}v_*.
\end{equation}
Perform the change of variable $V=v-v_*$ in \eqref{pdofnu}, then
\begin{equation*}
    \frac{\partial{\nu}}{\partial{v_i}}=\int\frac{I_*^{\alpha}(r(1-r))^{\alpha}(1-R)^{2\alpha+1}  R^{1/2}}{\Gamma(\alpha+1)(2\pi)^{\frac{3}{2}}} \frac{V_i}{|V|}\frac{\partial\mathcal{B}}{\partial|V|}(|V|,I,I_*,r,R,\omega)e^{-\frac{1}{2}(v-V)^2-I_*}\mathrm{d}r\mathrm{d}R\mathrm{d}\omega\mathrm{d}I_*\mathrm{d}V,
\end{equation*}
and thus, 
\begin{equation}\label{5.32}
    \sum_{i=1}^{3} v_i\frac{\partial{\nu}}{\partial{v_i}}=\int\frac{I_*^{\alpha}(r(1-r))^{\alpha}(1-R)^{2\alpha+1}  R^{1/2}}{\Gamma(\alpha+1)(2\pi)^{\frac{3}{2}}} \frac{v.V}{|V|}\frac{\partial\mathcal{B}}{\partial|V|}(|V|,I,I_*,r,R,\omega)e^{-\frac{1}{2}(v-V)^2-I_*}\mathrm{d}r\mathrm{d}R\mathrm{d}\omega\mathrm{d}I_*\mathrm{d}V.
\end{equation}
Applying Fubini theorem, we write \eqref{5.32} as
\begin{gather*}
    \sum_{i=1}^{3} v_i\frac{\partial{\nu}}{\partial{v_i}}=\int \bigg[\int (r(1-r))^{\alpha}(1-R)^{2\alpha+1}  R^{1/2}\frac{\partial\mathcal{B}}{\partial|V|}(|V|,I,I_*,r,R,\omega)\mathrm{d}r\mathrm{d}R\mathrm{d}\omega\bigg]\\\frac{I_*^{\alpha}}{\Gamma(\alpha+1)(2\pi)^{\frac{3}{2}}}\frac{v.V}{|V|}e^{-\frac{1}{2}(v-V)^2-I_*}\mathrm{d}I_*\mathrm{d}V.
\end{gather*}
\noindent The partial derivative of $\nu$ along $I$ is 
\begin{equation}\label{pdwrtI}
\begin{aligned}
 I\frac{\partial{\nu}}{\partial{I}}&=\int \frac{I_*^{\alpha}(r(1-r))^{\alpha}(1-R)^{2\alpha+1}  R^{1/2}}{\Gamma(\alpha+1)(2\pi)^{\frac{3}{2}}}I\frac{\partial {\mathcal{B}}}{\partial I}(|V|,I,I_*,r,R,\sigma)e^{-\frac{1}{2}(v-V)^2-I_*}\mathrm{d}r\mathrm{d}R\mathrm{d}\sigma\mathrm{d}I_*\mathrm{d}V\\
 &=\int \frac{I_*^{\alpha}}{\Gamma(\alpha+1)(2\pi)^{\frac{3}{2}}}I\Big[\int {(r(1-r))^{\alpha}(1-R)^{2\alpha+1}  R^{1/2}}\frac{\partial {\mathcal{B}}}{\partial I}(|V|,I,I_*,r,R,\sigma)\mathrm{d}r\mathrm{d}R\mathrm{d}\sigma\Big]e^{-\frac{1}{2}(v-V)^2-I_*}\mathrm{d}I_*\mathrm{d}V.
 \end{aligned}
\end{equation}
When $v.V>0,$ the exponential in the integral \eqref{5.32} is greater than when $v.V<0$, and so the term $v.V$ doesn't affect the sign of the partial derivatives of $\nu$. Therefore, the sign of the partial derivative of $\nu$ along  $|v|$ has the same sign as  
$$\int(r(1-r))^{\alpha}(1-R)^{2\alpha+1}  R^{1/2}\frac{\partial\mathcal{B}}{\partial|V|}(|V|,I,I_*,r,R,\omega)\mathrm{d}r\mathrm{d}R\mathrm{d}\omega.$$
It's clear as well that the partial derivative of $\nu$ with respect to $I$ \eqref{pdwrtI} has the same sign as
$$\int(r(1-r))^{\alpha}(1-R)^{2\alpha+1}  R^{1/2}\frac{\partial\mathcal{B}}{\partial I}(|V|,I,I_*,r,R,\omega)\mathrm{d}r\mathrm{d}R\mathrm{d}\omega.$$
As a result, for a collision cross-section $\mathcal{B}$ satisfying the condition that the integral 
   \begin{equation*}
    \int_{(0,1)^2\times S^2}(r(1-r))^{\alpha}(1-R)^{2\alpha+1}  R^{1/2} \mathcal{B}(|V|,I,I_*,r,R,\sigma) \mathrm{d}r\mathrm{d}R\mathrm{d}\sigma
\end{equation*}
 is increasing (respectively decreasing) in $|V|$ and $I$, the collision frequency is increasing (respectively decreasing).
\end{proof}

 \end{document}